\documentclass[10pt,reqno]{amsart}
\usepackage[T1]{fontenc}
\usepackage{lmodern}
\usepackage{microtype}
\usepackage[letterpaper,margin=1in]{geometry}
\usepackage{amsmath,amssymb,amsthm,mathtools}
\usepackage[hidelinks]{hyperref}
\hypersetup{pdftitle={The norm of the centered Hardy-Littlewood maximal operator},pdfauthor={Jose Madrid}}
\newtheorem{theorem}{Theorem}[section]
\newtheorem{proposition}[theorem]{Proposition}
\newtheorem{lemma}[theorem]{Lemma}
\newtheorem{corollary}[theorem]{Corollary}
\theoremstyle{remark}

\numberwithin{equation}{section}
\newcommand{\R}{\mathbb R}
\newcommand{\Z}{\mathbb Z}
\newcommand{\N}{\mathbb N}
\newcommand{\E}{\mathbb E}

\newcommand{\ind}{\mathbf 1}
\newcommand{\HLcenter}{M_{\mathrm c}}
\newcommand{\HLdyad}{M_{\mathrm{dyad}}}
\newcommand{\norm}[1]{\left\lVert #1\right\rVert}
\DeclareMathOperator{\dist}{dist}
\title[The centered maximal operator on the line]{The norm of the centered Hardy--Littlewood maximal operator}

\begin{document}

\author[Jos\'e Madrid]{Jos\'e Madrid}

\address{Department of Mathematics, Virginia Polytechnic Institute and State University, 225 Stanger Street, Blacksburg, VA 24061-1026, USA}
	\email{josemadrid@vt.edu}

\begin{abstract}
We determine the exact $L^p(\R)$ norm of the centered Hardy--Littlewood maximal operator for every $1<p<\infty$, proving that it equals $p/(p-1)$. This settles the sharp strong-type problem for the centered maximal operator on the real line. For every $p$ in this range, the norm is strictly larger than the constant associated with the critical power $|x|^{-1/p}$, thereby disproving the conjecture of Dror, Ganguli, and Strichartz that this profile determines the unrestricted norm.
\end{abstract}
\maketitle

\section{Introduction}

\subsection{The centered maximal inequality}
For $f\in L^1_{\mathrm{loc}}(\R)$, the centered Hardy--Littlewood maximal function is defined by
\begin{equation}\label{eq:defmax}
 \HLcenter f(x)=\sup_{r>0}\frac1{2r}\int_{x-r}^{x+r}|f(y)|\,dy.
\end{equation}
The Hardy--Littlewood maximal theorem \cite{HL} gives the boundedness of $\HLcenter$ on $L^p(\R)$ for $1<p\le\infty$, together with a weak-type estimate at $p=1$. The purpose of this paper is to determine the optimal constant in the strong-type inequality. We write
\begin{equation}\label{eq:defnorm}
 p'=\frac p{p-1},\qquad
 N_p=\sup_{0\ne f\in L^p(\R)}
       \frac{\norm{\HLcenter f}_{L^p(\R)}}{\norm{f}_{L^p(\R)}},
 \qquad 1<p<\infty.
\end{equation}
A classical consequence of the sharp one-sided maximal inequalities is
\begin{equation}\label{eq:classicalupper}
 N_p\le p'.
\end{equation}
Indeed, the left and right maximal operators have norm $p'$, while the centered maximal function is bounded pointwise by their arithmetic mean. The one-sided theorem originates in Hardy and Littlewood \cite{HL}; Riesz \cite{Riesz} gave the rising-sun argument that yields the sharp estimate. An explicit $L^p$ formulation appears in Phillips \cite[Theorem~3.2(i), p.~652]{Phillips}. We recall the argument in Section~\ref{sec:upper}.

The issue is sharpness. Centered intervals impose a common radius on the left and right averages, and the familiar examples built from a single power singularity give a smaller constant. We construct functions with values arranged on many separated scales and prove that their norm ratios approach the upper bound in \eqref{eq:classicalupper}.

\subsection{Main results}
Our main theorem determines the exact operator norm of $\HLcenter$ on $L^p(\R)$.

\begin{theorem}\label{thm:main}
Let $1<p<\infty$. Then
\begin{equation}\label{eq:main}
 \norm{\HLcenter}_{L^p(\R)\to L^p(\R)}=\frac p{p-1}.
\end{equation}
For every $\varepsilon>0$, there is a nonzero, nonnegative $f\in C_c^\infty(\R)$ such that
\begin{equation}\label{eq:smoothsharp}
 \norm{\HLcenter f}_p>\left(\frac p{p-1}-\varepsilon\right)\norm{f}_p.
\end{equation}
\end{theorem}

The extremizing functions are constructed from finite martingales. Their values occur on many separated spatial scales, and finitely many centered averages recover the successive conditional means. The construction gives a concrete sequence $f_B$, indexed by integers $B\ge2$, for which
\[
 \lim_{B\to\infty}\frac{\norm{\HLcenter f_B}_p}{\norm{f_B}_p}=p'.
\]
The formula for $f_B$ and the finite lower bound underlying this limit are given in Section~\ref{sec:proof}.

Only a sparse set of radii is needed. Define
\begin{equation}\label{eq:dyadic}
 \HLdyad f(x)=\sup_{j\in\Z}\frac1{2^{j+1}}
   \int_{x-2^j}^{x+2^j}|f(y)|\,dy.
\end{equation}
Here the intervals are centered at $x$ and their radii are powers of two.

\begin{corollary}\label{cor:dyadic}
For every $1<p<\infty$,
\begin{equation}\label{eq:dyadicnorm}
 \norm{\HLdyad}_{L^p(\R)\to L^p(\R)}=\frac p{p-1}.
\end{equation}
\end{corollary}

Thus restricting the available radii to this geometric sequence leaves the norm unchanged. The proof uses the same functions as Theorem~\ref{thm:main}, with the digit base restricted to powers of two.

\subsection{Historical background}\label{subsec:history}

For the uncentered maximal operator $M_{\mathrm u}$, Grafakos and Montgomery-Smith \cite{GM} found that the optimal $L^p(\R)$ constant $U_p$ is the unique positive root of
\begin{equation}\label{eq:uncentered}
 (p-1)U_p^p-pU_p^{p-1}-1=0.
\end{equation}
The elementary pointwise comparison $\tfrac12 M_{\mathrm u}f\le\HLcenter f\le M_{\mathrm u}f$ consequently gives $\tfrac12 U_p\le N_p\le U_p$. The upper bound \eqref{eq:classicalupper} is smaller, since $U_p>p'$.

Grafakos and Kinnunen \cite[Corollary~2.3]{GK} extended the uncentered estimate with constant $U_p$ to locally finite Borel measures on the line. This estimate is sharp uniformly over the measures, since Lebesgue measure gives equality at the level of operator norms. They also obtained centered estimates in every dimension \cite[Theorem~3.2 and Corollary~3.3]{GK}. Their bound on $L^p(\R^d,\mu)$ is $A_{p,d}$, where
\begin{equation}\label{eq:GKcentered}
 (p-1)A_{p,d}^p-pA_{p,d}^{p-1}-(\beta_d-1)=0
\end{equation}
and $\beta_d$ is a Besicovitch covering constant. In dimension one one may take $\beta_1=2$, which gives $A_{p,1}=U_p$. These results address estimates uniform over general measures; the bound $p'$ in \eqref{eq:classicalupper} uses the equal Lebesgue lengths of the two halves of a centered interval.

A more specific lower bound comes from homogeneous functions. Christ and Grafakos \cite[Section~3]{CG} and, independently, Dror, Ganguli, and Strichartz \cite[Theorem~3.2]{DGS} proved that $N_p\ge c_p$, where
\begin{equation}\label{eq:peak}
 c_p=p'\sup_{t>1}\frac{(t+1)^{1-1/p}+(t-1)^{1-1/p}}{2t}.
\end{equation}
This constant is characterized by $\HLcenter(|x|^{-1/p})=c_p|x|^{-1/p}$. Although the power itself is outside $L^p(\R)$, truncations give admissible functions whose norm ratios approach $c_p$. Dror, Ganguli, and Strichartz combined this lower bound with a numerical search and conjectured that $N_p=c_p$.

Grafakos, Montgomery-Smith, and Motrunich \cite{GMM} proved the conjectured inequality for nonnegative functions convex on each side of one point, with optimal constant $c_p$. Their proof uses a variational argument and the structure of the radius at which the maximal average is attained. For arbitrary functions, discontinuities of this radius introduce additional terms in the changes of variables, and their argument does not control these terms. Their theorem therefore gives a sharp inequality on the convex class, together with a precise indication of the difficulty in extending that approach.

The estimates above leave the interval
\begin{equation}\label{eq:earliergap}
 c_p\le N_p\le p'.
\end{equation}
In particular, the calculation in \cite[Lemma~3.3]{DGS} gives $c_2=27^{1/4}/\sqrt2$, so that
\begin{equation}\label{eq:p2gap}
 \frac{27^{1/4}}{\sqrt2}\le N_2\le2.
\end{equation}
Theorem~\ref{thm:main} identifies the upper endpoint in \eqref{eq:earliergap} as the norm for every $p$. The strict inequality $c_p<p'$ is proved in the final section.

The centered weak-type $(1,1)$ problem has a separate history. Dror, Ganguli, and Strichartz \cite{DGS} conjectured that its optimal constant $W$ is $3/2$. Aldaz \cite{Aldaz} disproved that value and obtained
\[
 \frac{37}{24}\le W\le\frac{9+\sqrt{41}}8.
\]
Melas \cite{Melas} eventually determined the exact constant
\[
 W=\frac{11+\sqrt{61}}{12}.
\]
These endpoint estimates concern the measures of level sets. The exact strong-type norm remained undetermined, as recorded in Ciccone and Wr\'obel \cite[Section~1.1]{CW}. Our proof obtains it by realizing finite conditional expectations at selected centered radii.

\subsection{Martingales and centered averages}
The lower bound rests on a realization principle. Let $F\ge0$ be a nonzero function on the uniform probability space $\Omega=\{0,\ldots,B-1\}^n$, and let $F_k$ denote its conditional expectation given the first $k$ coordinates. Lemma~\ref{lem:realization} shows that
\[
 N_p\ge
 \frac{\bigl\|\max_{0\le k\le n}F_k\bigr\|_{L^p(\Omega)}}
      {\norm{F}_{L^p(\Omega)}}.
\]
The same constant $p'$ is the optimal constant in Doob's
$L^p$ martingale maximal inequality \cite{Doob}. We use a
family of finite stopped martingales to approach this constant,
computing their moments explicitly in
Section~\ref{sec:martingale}.

To obtain the realization, we encode the coordinates of $F$ at widely separated positions in the base-$B$ expansion of a point $x\in[0,1)$. The resulting periodic function has exactly the distribution of $F$. For each $k$, we choose a radius small enough to preserve the first $k$ selected digits at most centers, but large enough to average the remaining digits over an integer number of periods. At every such center, the centered average is precisely $F_k$.

The distribution of the admissible centers is the useful additional feature of the construction. The conditions that keep an averaging interval inside a region of constant past digits can be expressed using blocks of unused digits. These blocks are disjoint from the selected coordinates and from one another. Counting finite base-$B$ words therefore shows that every selected configuration retains exactly the same proportion of centers. This gives a multiplicative estimate whose relative loss is independent of the size of $F$.

That independence allows the martingale depth and the spatial parameters to vary together. We take depth $n=B^2$ and a fixed spacing between selected digit positions, then truncate the periodic function on a sufficiently long interval. Both spatial losses tend to zero as $B\to\infty$, while the martingale ratio tends to $p'$. This produces the explicit extremizing sequence. Approximation in $L^p$ then gives the smooth functions in Theorem~\ref{thm:main}.

Section~\ref{sec:upper} recalls the upper bound. Section~\ref{sec:realization} proves the realization lemma, including its exact distribution identity and a two-coordinate example. Section~\ref{sec:martingale} constructs the finite martingale, and Section~\ref{sec:proof} combines the two ingredients. The final section proves $c_p<p'$ and relates the result to the earlier critical-power examples.

\subsection{Notation and conventions}
Throughout, $1<p<\infty$ is fixed unless otherwise stated. Integrals on $\R$ are taken with respect to Lebesgue measure, $|E|$ denotes the measure of a measurable set, and $\N=\{1,2,\ldots\}$. On finite probability spaces, expectations refer to the specified uniform measure. For a nonnegative function $h$, write
\begin{equation}\label{eq:average}
 A_rh(x)=\frac1{2r}\int_{x-r}^{x+r}h(y)\,dy.
\end{equation}
When $h$ is initially defined on $[0,1)$, this average uses its periodic extension to $\R$. We write $\{t\}=t-\lfloor t\rfloor$ for the fractional part. Digits and periodic step functions are defined on half-open intervals; the values at the endpoints of a compact-support cutoff have no effect on the averages or norms.

\section{The one-sided upper bound}\label{sec:upper}

We record the classical upper estimate to fix its attribution and make the proof self-contained. The sharp one-sided inequality is stated in \cite[Theorem~3.2(i)]{Phillips}; the argument below uses Riesz's rising-sun method \cite{Riesz}.

For $f\ge0$, define
\[
 M_+f(x)=\sup_{r>0}\frac1r\int_x^{x+r}f(y)\,dy,
 \qquad
 M_-f(x)=\sup_{r>0}\frac1r\int_{x-r}^{x}f(y)\,dy.
\]

\begin{proposition}\label{prop:upper}
For every $f\in L^p(\R)$,
\begin{equation}\label{eq:upper}
 \norm{\HLcenter f}_p\le p'\norm{f}_p.
\end{equation}
\end{proposition}

\begin{proof}
First assume that $f\ge0$ is bounded and compactly supported. Riesz's rising-sun argument gives the following one-sided level-set identity (in this form, see \cite[Lemma~1]{GM}):
\begin{equation}\label{eq:sunrise}
 \lambda|\{M_\pm f>\lambda\}|
   =\int_{\{M_\pm f>\lambda\}}f(x)\,dx,
 \qquad \lambda>0.
\end{equation}
Layer cake, Tonelli's theorem, and H\"older's inequality give
\begin{align*}
 \norm{M_\pm f}_p^p
 &=p\int_0^\infty\lambda^{p-2}
       \int_{\{M_\pm f>\lambda\}}f(x)\,dx\,d\lambda\\
 &=p'\int_\R f(x)(M_\pm f(x))^{p-1}\,dx\\
 &\le p'\norm{f}_p\norm{M_\pm f}_p^{p-1}.
\end{align*}
The norms are finite because $M_\pm f$ is bounded and has an $O(|x|^{-1})$ tail. Hence $\norm{M_\pm f}_p\le p'\norm{f}_p$. Since
\[
 A_rf(x)=\frac12\left(\frac1r\int_{x-r}^{x}f(y)\,dy
                    +\frac1r\int_x^{x+r}f(y)\,dy\right),
\]
we have $\HLcenter f\le(M_-f+M_+f)/2$, proving \eqref{eq:upper}. For general $f$, apply the estimate to
$f_m=\min\{|f|,m\}\ind_{[-m,m]}$ and use monotone convergence, since $\HLcenter f_m\uparrow\HLcenter f$.
\end{proof}

\section{Finite martingales and centered averages}\label{sec:realization}

We begin with a localization observation, and then construct periodic functions whose centered averages reproduce a prescribed finite martingale. The separation between successive selected digits provides the two scales needed for this construction: the first coordinates remain fixed while the later coordinates are averaged.

\subsection{Localization of periodic examples}
The construction is most naturally carried out with periodic functions. The following lemma transfers a lower bound involving finitely many periodic averages to a compactly supported function on the line.

\begin{lemma}\label{lem:localization}
Let $h$ be a bounded, nonnegative, one-periodic function with $\int_0^1h^p>0$. Let $r_0,\ldots,r_n$ be positive radii such that $r_0=\max_k r_k\in\N$. For an integer $R>r_0$, put
\[
 f_R=h\ind_{[-R,R]},\qquad S(x)=\max_{0\le k\le n}A_{r_k}h(x).
\]
Then
\begin{equation}\label{eq:localization}
 \frac{\norm{\HLcenter f_R}_p^p}{\norm{f_R}_p^p}
 \ge\left(1-\frac{r_0}{R}\right)
   \frac{\displaystyle\int_0^1S(x)^p\,dx}
        {\displaystyle\int_0^1h(x)^p\,dx}.
\end{equation}
\end{lemma}

\begin{proof}
If $x\in[-R+r_0,R-r_0]$ and $0\le k\le n$, then
$[x-r_k,x+r_k]\subset[-R,R]$. Thus $A_{r_k}f_R(x)=A_{r_k}h(x)$, and consequently $\HLcenter f_R(x)\ge S(x)$ on this interior interval. Since $S$ is one-periodic and both endpoints are integers,
\begin{equation}\label{eq:localnumerator}
 \norm{\HLcenter f_R}_p^p
 \ge\int_{-R+r_0}^{R-r_0}S(x)^p\,dx
 =2(R-r_0)\int_0^1S(x)^p\,dx.
\end{equation}
Similarly,
\begin{equation}\label{eq:localdenominator}
 \norm{f_R}_p^p=\int_{-R}^{R}h(x)^p\,dx
 =2R\int_0^1h(x)^p\,dx>0.
\end{equation}
Dividing \eqref{eq:localnumerator} by \eqref{eq:localdenominator} proves the lemma.
\end{proof}

\subsection{The finite realization lemma}
Fix integers $B\ge2$ and $n\ge1$, and write
\[
 D_B=\{0,1,\ldots,B-1\},\qquad\Omega=D_B^n.
\]
The probability of each point of $\Omega$ is $B^{-n}$. Let $F:\Omega\to[0,\infty)$ be nonzero. For $0\le k\le n$, define
\begin{equation}\label{eq:conditional}
 F_k(u_1,\ldots,u_k)
 =\frac1{B^{n-k}}
   \sum_{w\in D_B^{n-k}}F(u_1,\ldots,u_k,w).
\end{equation}
For $k=0$ this is the constant $F_0=\E F$, and for $k=n$ the sum has one term and $F_n=F$. Equivalently, $F_k$ is the conditional expectation of $F$ given the first $k$ coordinates. In particular, $0\le F_k\le H$, where $H=\max_\Omega F$. Set
\begin{equation}\label{eq:finite-star}
 F_*(u_1,\ldots,u_n)=\max_{0\le k\le n}F_k(u_1,\ldots,u_k).
\end{equation}

\begin{lemma}[Finite realization]\label{lem:realization}
Let $1<p<\infty$ and let $B,n,F$ be as above. For every integer $s\ge1$ and every integer $R>B^s$, there is a nonzero, nonnegative, compactly supported step function $f_{s,R}$ such that
\begin{equation}\label{eq:additive}
 \frac{\norm{\HLcenter f_{s,R}}_p^p}{\norm{f_{s,R}}_p^p}
 \ge\left(1-\frac{B^s}{R}\right)
   \frac{\E F_*^p-2nB^{1-s}H^p}{\E F^p}.
\end{equation}
For the same function, if $s\ge2$, the stronger estimate
\begin{equation}\label{eq:multiplicative}
 \frac{\norm{\HLcenter f_{s,R}}_p^p}{\norm{f_{s,R}}_p^p}
 \ge\left(1-\frac{B^s}{R}\right)
   \bigl(1-2B^{1-s}\bigr)^n\frac{\E F_*^p}{\E F^p}
\end{equation}
holds. Consequently,
\begin{equation}\label{eq:transfer}
 N_p\ge\frac{(\E F_*^p)^{1/p}}{(\E F^p)^{1/p}}.
\end{equation}
\end{lemma}

Observe that $\E F^p>0$, so every quotient in the statement is defined. The right-hand side of \eqref{eq:additive} may be negative for small $s$; the inequality is still valid in that case. In \eqref{eq:multiplicative}, $s\ge2$ implies $0\le1-2B^{1-s}<1$. This factor vanishes only when $B=2$ and $s=2$. To obtain a useful positive lower bound, one may take $s\ge3$.

\subsection{Proof of Lemma \ref{lem:realization}}
We first construct the periodic function and the averaging radii, and prove the conditional-average identity. We then estimate the set where these identities hold simultaneously. The additive estimate follows from a union bound; the multiplicative estimate follows from an exact count of the unused digits. Finally, we truncate the periodic function.

\begin{proof}
\emph{Step 1. Ordinary digits and the selected coordinates.}
For $\ell\ge1$ and $x\in[0,1)$, define the ordinary base-$B$ digit
\begin{equation}\label{eq:ordinarydigits}
 e_\ell(x)=\lfloor B^\ell x\rfloor-B\lfloor B^{\ell-1}x\rfloor.
\end{equation}
It belongs to $D_B$. This formula specifies the digit at every point, including points with a terminating expansion.

Fix $s\ge1$ and put
\begin{equation}\label{eq:scaleparameters}
 L=B^{2s},\qquad P_j=1+2s(j-1),\qquad 1\le j\le n.
\end{equation}
For all real $x$, set
\begin{equation}\label{eq:selecteddigits}
 d_j(x)=\lfloor B\{L^{j-1}x\}\rfloor,
 \qquad h_s(x)=F(d_1(x),\ldots,d_n(x)).
\end{equation}
Since
\[
 \lfloor B\{t\}\rfloor=\lfloor Bt\rfloor-B\lfloor t\rfloor,
\]
we have $d_j(x)=e_{P_j}(x)$ for $0\le x<1$. Thus the selected positions are
\[
 P_1=1,\quad P_2=2s+1,\quad\ldots,\quad P_n=2s(n-1)+1.
\]
Because $L$ is an integer, $d_j(x+1)=d_j(x)$ for every real $x$. It follows that $h_s$ is one-periodic. It is a step function with finitely many discontinuities per period, and $0\le h_s\le H$.

\smallskip
\emph{Step 2. Exact counting of the selected digit vectors.}
Let $m=P_n$, and partition $[0,1)$ into the $B^m$ half-open intervals
\begin{equation}\label{eq:cylinders}
 I_a=[aB^{-m},(a+1)B^{-m}),\qquad 0\le a<B^m.
\end{equation}
For $1\le\ell\le m$ and $x\in I_a$,
\begin{equation}\label{eq:digitoncell}
 e_\ell(x)=
 \left\lfloor\frac{a}{B^{m-\ell}}\right\rfloor
 -B\left\lfloor\frac{a}{B^{m-\ell+1}}\right\rfloor.
\end{equation}
Indeed, writing $x=(a+t)B^{-m}$ with $0\le t<1$, the addition of $t$ cannot change either of the two integer quotients in \eqref{eq:digitoncell}. The resulting vector $(e_1,\ldots,e_m)$ is precisely the length-$m$ base-$B$ word representing $a$, with leading zeros allowed.

Fix $u=(u_1,\ldots,u_n)\in D_B^n$. Imposing $d_j=u_j$ fixes the digits at the $n$ distinct positions $P_j$. The other $m-n$ digits may be chosen arbitrarily. Therefore exactly $B^{m-n}$ intervals in \eqref{eq:cylinders} have this selected vector, and
\begin{equation}\label{eq:uniformlaw}
 \bigl|\{x\in[0,1):(d_1(x),\ldots,d_n(x))=u\}\bigr|
 =B^{m-n}B^{-m}=B^{-n}.
\end{equation}
Thus the selected digits have the uniform product distribution on $D_B^n$.

In particular, for every real-valued function $\Phi$ on $D_B^n$,
\begin{equation}\label{eq:lawphi}
 \int_0^1\Phi(d_1(x),\ldots,d_n(x))\,dx
 =\frac1{B^n}\sum_{u\in D_B^n}\Phi(u)=\E\Phi.
\end{equation}
Define
\begin{equation}\label{eq:Y}
 Y_k(x)=F_k(d_1(x),\ldots,d_k(x)),\qquad
 Y_*(x)=\max_{0\le k\le n}Y_k(x),
\end{equation}
with $Y_0=\E F$. Applying \eqref{eq:lawphi} to $F^p$ and $F_*^p$ gives
\begin{equation}\label{eq:exactlaw}
 \int_0^1h_s(x)^p\,dx=\E F^p,\qquad
 \int_0^1Y_*(x)^p\,dx=\E F_*^p.
\end{equation}
The second equality uses the joint distribution of all selected digits, and hence of all conditional expectations simultaneously.

\smallskip
\emph{Step 3. The averaging radii and the past-coordinate grids.}
Choose
\begin{equation}\label{eq:radii}
 r_k=B^sL^{-k},\qquad 0\le k\le n.
\end{equation}
These radii are positive and decreasing, and their largest value is the integer $r_0=B^s$.
For $1\le k\le n$, put
\begin{equation}\label{eq:grids}
 \Delta_k=(BL^{k-1})^{-1}=B^{-P_k},\qquad
 \mathcal G_k=\Delta_k\Z.
\end{equation}
The discontinuities of $d_j$ lie in $(BL^{j-1})^{-1}\Z$. If $j\le k$, this grid is contained in $\mathcal G_k$, since $L^{k-j}$ is an integer. Thus all of the first $k$ selected digits are constant on every connected component of $\R\setminus\mathcal G_k$.

The two scale relations used below are
\begin{equation}\label{eq:scalerelations}
 \frac{r_k}{\Delta_k}=B^{1-s},\qquad
 \frac{2r_k}{L^{-k}}=2B^s\in\N.
\end{equation}
The first ratio tends to zero as $s$ increases, making $r_k$ small compared with a cell where the past digits are constant. The second makes the averaging interval an integer number of periods of the future-coordinate function. Both relations hold simultaneously because unused digit positions separate the selected coordinates.

Define the exceptional set
\begin{equation}\label{eq:badsets}
 E_k=\{x\in[0,1):\dist(x,\mathcal G_k)\le r_k\},
 \qquad 1\le k\le n.
\end{equation}
There are $BL^{k-1}$ grid points per unit period. By counting their radius-$r_k$ neighborhoods on $\R/\Z$,
\begin{equation}\label{eq:badmeasure}
 |E_k|\le2r_kBL^{k-1}=2B^{1-s}.
\end{equation}
If some neighborhoods overlap, the same upper bound holds. The definition with $\le r_k$ ensures that for $x\in[0,1)\setminus E_k$ the entire closed interval $[x-r_k,x+r_k]$ avoids every past-coordinate discontinuity.

\smallskip
\emph{Step 4. Freezing the past coordinates.}
Fix $1\le k\le n$ and $x\in[0,1)\setminus E_k$. Let
$v_j=d_j(x)$ for $1\le j\le k$, and define the auxiliary function
\begin{equation}\label{eq:auxiliary}
 g_{k,v}(y)=F(v_1,\ldots,v_k,d_{k+1}(y),\ldots,d_n(y)),
 \qquad y\in\R.
\end{equation}
Its first $k$ arguments have been fixed globally. On the selected averaging interval, however, the actual first $k$ digits are constant with precisely these values. Consequently,
\begin{equation}\label{eq:localagreement}
 h_s(y)=g_{k,v}(y),\qquad x-r_k\le y\le x+r_k.
\end{equation}
This equality is the reason we can average the auxiliary function in place of $h_s$.

For $j>k$,
\[
 L^{j-1}(y+L^{-k})=L^{j-1}y+L^{j-k-1}.
\]
The added term is an integer, so $d_j(y+L^{-k})=d_j(y)$. Thus $g_{k,v}$ has period
\begin{equation}\label{eq:futureperiod}
 \tau_k=L^{-k}.
\end{equation}
The fixed past coordinates allow this shorter period, while \eqref{eq:localagreement} permits us to use it when computing the centered average of $h_s$ on the chosen interval.

\smallskip
\emph{Step 5. The mean over one future-coordinate period.}
We claim that
\begin{equation}\label{eq:periodmean}
 \frac1{\tau_k}\int_0^{\tau_k}g_{k,v}(y)\,dy=F_k(v_1,\ldots,v_k).
\end{equation}
For $k<n$, substitute $u=L^ky$. For $1\le\ell\le n-k$,
\begin{equation}\label{eq:futurechange}
 d_{k+\ell}(L^{-k}u)=\lfloor B\{L^{\ell-1}u\}\rfloor.
\end{equation}
As $u$ runs over $[0,1)$, the right-hand sides are the same selected-digit construction, now with $n-k$ coordinates. By Step 2, their joint distribution is uniform on $D_B^{n-k}$. Therefore
\begin{align}
 \frac1{\tau_k}\int_0^{\tau_k}g_{k,v}(y)\,dy
 &=\int_0^1F\bigl(v_1,\ldots,v_k,
       \lfloor B\{u\}\rfloor,\ldots,
       \lfloor B\{L^{n-k-1}u\}\rfloor\bigr)\,du\notag\\
 &=\frac1{B^{n-k}}\sum_{w\in D_B^{n-k}}F(v_1,\ldots,v_k,w)
 =F_k(v_1,\ldots,v_k).\label{eq:periodmeancalculation}
\end{align}
When $k=n$, the auxiliary function is the constant $F(v)$, and \eqref{eq:periodmean} holds as well.

\smallskip
\emph{Step 6. Exact averaging with arbitrary interval endpoints.}
A periodic function has the same integral over every interval of one period. To make the treatment of the averaging endpoints explicit, let $g$ be locally integrable with period $\tau>0$, and write $a=m\tau+b$ with $m\in\Z$ and $0\le b<\tau$. Then
\begin{equation}\label{eq:phase}
 \int_a^{a+\tau}g(y)\,dy
 =\int_b^\tau g(y)\,dy+\int_0^b g(y)\,dy
 =\int_0^\tau g(y)\,dy.
\end{equation}
Splitting an interval of length $N\tau$ into $N$ consecutive intervals of length $\tau$, for $N\in\N$, yields
\begin{equation}\label{eq:manyperiods}
 \frac1{N\tau}\int_a^{a+N\tau}g(y)\,dy
 =\frac1\tau\int_0^\tau g(y)\,dy.
\end{equation}
The initial point $a$ need not lie on a period boundary.

Apply \eqref{eq:manyperiods} to $g_{k,v}$ with $\tau=\tau_k=L^{-k}$, $N=2B^s$, and $a=x-r_k$. Combining \eqref{eq:localagreement}, \eqref{eq:scalerelations}, and \eqref{eq:periodmean}, we obtain
\begin{equation}\label{eq:exactaverages}
 A_{r_k}h_s(x)=F_k(v_1,\ldots,v_k)=Y_k(x),
 \qquad x\in[0,1)\setminus E_k.
\end{equation}
For $k=0$, no exceptional set is required. The function $h_s$ is one-periodic, and the interval length $2r_0=2B^s$ is an integer. Therefore
\begin{equation}\label{eq:zerocoordinate}
 A_{r_0}h_s(x)=\int_0^1h_s(y)\,dy=\E F=Y_0,
 \qquad x\in\R.
\end{equation}

\smallskip
\emph{Step 7. Simultaneous realization and the additive estimate.}
Set
\begin{equation}\label{eq:goodset}
 E=\bigcup_{k=1}^nE_k,\qquad G=[0,1)\setminus E,
 \qquad S_s(x)=\max_{0\le k\le n}A_{r_k}h_s(x).
\end{equation}
All identities \eqref{eq:exactaverages} hold simultaneously on $G$, and \eqref{eq:zerocoordinate} holds everywhere. Thus
\begin{equation}\label{eq:simultaneous}
 S_s(x)=Y_*(x),\qquad x\in G.
\end{equation}
Since $0\le Y_*\le H$, the union bound and \eqref{eq:exactlaw} give
\begin{align}
 \int_0^1S_s(x)^p\,dx
 &\ge\int_GY_*(x)^p\,dx\notag\\
 &=\E F_*^p-\int_EY_*(x)^p\,dx\notag\\
 &\ge\E F_*^p-H^p\sum_{k=1}^n|E_k|
 \ge\E F_*^p-2nB^{1-s}H^p.\label{eq:additiveperiodic}
\end{align}
This part of the argument requires no independence between the exceptional set and the selected digits. It is valid for every $s\ge1$.

\smallskip
\emph{Step 8. A digit description of the good set.}
For the stronger estimate, assume from now on that $s\ge2$ and put
\begin{equation}\label{eq:guardQ}
 Q=B^{s-1}\ge2,\qquad\eta=Q^{-1}=B^{1-s}.
\end{equation}
Since $\mathcal G_k=B^{-P_k}\Z$ and $B^{P_k}r_k=\eta$, scaling the distance in \eqref{eq:badsets} gives
\begin{equation}\label{eq:scaledbad}
 x\notin E_k
 \quad\Longleftrightarrow\quad
 \dist(B^{P_k}x,\Z)>\eta
 \quad\Longleftrightarrow\quad
 \eta<\{B^{P_k}x\}<1-\eta.
\end{equation}
For $Q=2$, this condition is empty. Otherwise it removes the first and last subintervals of length $Q^{-1}$ in $[0,1)$.

The block of $s-1$ ordinary digits immediately following position $P_k$ has integer value
\begin{equation}\label{eq:blockvalue}
 a_k(x)=\left\lfloor Q\{B^{P_k}x\}\right\rfloor
 =\sum_{j=1}^{s-1}e_{P_k+j}(x)B^{s-1-j}.
\end{equation}
The equivalence
\begin{equation}\label{eq:allowedblock}
 1\le a_k(x)\le Q-2
 \quad\Longleftrightarrow\quad
 Q^{-1}\le\{B^{P_k}x\}<1-Q^{-1}
\end{equation}
is exact. When $Q>2$, it differs from \eqref{eq:scaledbad} only at points with $\{B^{P_k}x\}=Q^{-1}$. There are finitely many such points in $[0,1)$, all on base-$B$ interval boundaries. When $Q=2$, both good conditions are empty.

Consequently, up to a finite set of endpoints, $G$ is the set where every block
\begin{equation}\label{eq:blockpositions}
 \{P_k+1,\ldots,P_k+s-1\},\qquad 1\le k\le n,
\end{equation}
contains neither the all-zero word nor the all-$(B-1)$ word. Each block has $Q$ possible words, of which exactly $Q-2$ are allowed.

These blocks are pairwise disjoint. Indeed, $P_{k+1}-P_k=2s$, whereas each block has length $s-1$. They are also disjoint from every selected position $P_j$. Thus the conditions defining the good set involve only unused coordinates of the base-$B$ expansion. The last block, following $P_n$, must be included; it controls the centers used for the radius $r_n$.

\smallskip
\emph{Step 9. Exact counting on the good set.}
To include all selected digits and all blocks in \eqref{eq:blockpositions}, refine the partition to depth
\begin{equation}\label{eq:refineddepth}
 M=P_n+s-1=s(2n-1).
\end{equation}
As in Step 2, the $B^M$ half-open intervals of length $B^{-M}$ correspond exactly to the words in $D_B^M$. The endpoint differences just described lie on boundaries of these intervals and have measure zero.

Fix a selected vector $u\in D_B^n$. Its $n$ coordinates occupy the positions $P_k$. The $n$ blocks occupy another $n(s-1)$ positions. Together these account for $ns$ positions, leaving
\begin{equation}\label{eq:freepositions}
 M-ns=s(n-1)
\end{equation}
free digits. For each fixed $u$, the number of words satisfying every allowed-block condition is therefore
\begin{equation}\label{eq:goodcount}
 B^{s(n-1)}(Q-2)^n.
\end{equation}
The factor $B^{s(n-1)}$ counts the unrestricted positions, and each of the $n$ disjoint blocks contributes $Q-2$ choices. Multiplying by the cell length $B^{-M}$ yields
\begin{align}
 \bigl|G\cap\{(d_1,\ldots,d_n)=u\}\bigr|
 &=B^{-M}B^{s(n-1)}(Q-2)^n\notag\\
 &=B^{-ns}(Q-2)^n
 =B^{-n}\left(1-\frac2Q\right)^n.\label{eq:goodjointlaw}
\end{align}
Here we used $Q=B^{s-1}$ in the last equality. In particular,
\begin{equation}\label{eq:goodmeasure}
 |G|=\left(1-\frac2Q\right)^n=(1-2B^{1-s})^n.
\end{equation}
More generally, multiplying \eqref{eq:goodjointlaw} by $\Phi(u)$ and summing gives the weighted identity
\begin{equation}\label{eq:weightedgood}
 \int_G\Phi(d_1(x),\ldots,d_n(x))\,dx
 =(1-2B^{1-s})^n\E\Phi
\end{equation}
for every real-valued function $\Phi$ on $D_B^n$. Thus the good set is independent of the entire selected digit vector, in the precise sense of \eqref{eq:goodjointlaw}.

Apply \eqref{eq:weightedgood} to $\Phi=F_*^p$. Using \eqref{eq:simultaneous}, we obtain
\begin{equation}\label{eq:multiplicativeperiodic}
 \int_0^1S_s(x)^p\,dx
 \ge\int_GY_*(x)^p\,dx
 =(1-2B^{1-s})^n\E F_*^p.
\end{equation}
The distribution identity applies to $Y_*^p=F_*^p(d_1,\ldots,d_n)$; its equality with $S_s^p$ on $G$ gives the required lower bound.

\smallskip
\emph{Step 10. Compact support and the passage to the norm.}
Set
\begin{equation}\label{eq:cutoff}
 f_{s,R}(x)=h_s(x)\ind_{[-R,R]}(x).
\end{equation}
It is a nonnegative step function with finitely many discontinuities on its compact support. Equation \eqref{eq:exactlaw} and the integrality of $R$ give
\begin{equation}\label{eq:exactdenominator}
 \norm{f_{s,R}}_p^p=2R\E F^p>0.
\end{equation}
The radii in \eqref{eq:radii} satisfy all hypotheses of Lemma~\ref{lem:localization}, since their maximum is the integer $B^s$. Applying that lemma first with \eqref{eq:additiveperiodic} and then with \eqref{eq:multiplicativeperiodic} proves \eqref{eq:additive} and \eqref{eq:multiplicative}, respectively.

Since each $f_{s,R}$ is admissible in the definition of $N_p$,
\begin{equation}\label{eq:normintermediate}
 N_p^p\ge
 \left(1-\frac{B^s}{R}\right)(1-2B^{1-s})^n
 \frac{\E F_*^p}{\E F^p},\qquad s\ge2.
\end{equation}
For every fixed $B,n,F,s$, let $R\to\infty$ through integers. Then let $s\to\infty$, holding $B,n,F$ fixed. The two spatial factors tend to one, and we obtain
\[
 N_p^p\ge\frac{\E F_*^p}{\E F^p}.
\]
Taking $p$-th roots proves \eqref{eq:transfer}. These limits compare norm ratios of different admissible functions; no convergence of the functions themselves is required.
\end{proof}

\subsection{A two-coordinate example}
We illustrate the geometry of the lemma without using the example as an input to the proof. Take $B=2$, $s=3$, and $n=2$. Then $L=64$, the selected positions are $P_1=1$ and $P_2=7$, and
\[
 d_1(x)=\lfloor2\{x\}\rfloor,\qquad
 d_2(x)=\lfloor2\{64x\}\rfloor.
\]
Write the four values of $F$ as $c_{ij}=F(i,j)$, $i,j\in\{0,1\}$. Its conditional expectations are
\[
 F_0=\frac{c_{00}+c_{01}+c_{10}+c_{11}}4,
 \qquad F_1(i)=\frac{c_{i0}+c_{i1}}2,
 \qquad F_2(i,j)=c_{ij}.
\]
The selected radii are
\[
 r_0=8,\qquad r_1=\frac18,\qquad r_2=\frac1{512}.
\]
For the middle radius, the first digit is constant on each half-unit interval. The admissible centers are
\[
 x\in\left(\frac18,\frac38\right)
 \cup\left(\frac58,\frac78\right).
\]
At such a center the interval of length $2r_1=1/4$ remains in one half-unit interval and contains exactly sixteen periods of length $1/64$ of the second-digit function. Its average is therefore $F_1(d_1(x))$.

For the smallest radius, both selected digits are constant between points of $128^{-1}\Z$. Away from the $1/512$-neighborhood of that grid, $A_{r_2}h_s=h_s=F_2(d_1,d_2)$. The largest radius gives the global mean $F_0$ everywhere.

The good-set conditions use the digit blocks at positions $\{2,3\}$ and $\{8,9\}$. Each must equal either $01$ or $10$. These conditions leave the selected digits at positions $1$ and $7$ unrestricted. They give $|G|=(1/2)^2=1/4$, and each fixed selected pair occupies measure $1/16$ inside $G$. The block after the last selected digit is visible in this example: positions $8$ and $9$ are needed even though the values of $h_s$ depend only on positions $1$ and $7$.

\subsection{Comparison of the two estimates}
The multiplicative estimate \eqref{eq:multiplicative} is stronger than \eqref{eq:additive}. Indeed, for $s\ge2$, Bernoulli's inequality gives
\[
 (1-2B^{1-s})^n\ge1-2nB^{1-s},
\]
while $\E F_*^p\le H^p$. Hence the periodic lower bound in \eqref{eq:multiplicativeperiodic} is at least the one in \eqref{eq:additiveperiodic}. The relative loss in \eqref{eq:multiplicativeperiodic} depends only on $B$, $n$, and $s$. The estimate therefore remains useful when the terminal values of the martingale become large. We use this freedom below to let the depth grow with the base.

\section{A finite martingale with explicit moments}\label{sec:martingale}

We next choose the finite martingale to be realized on the line. The process increases by a fixed factor while successes continue, drops at the first failure, and then remains constant. Its growth rate is chosen so that each possible first-failure time makes the same contribution to each of the relevant $p$-th moments.

\subsection{Definition of the terminal value}
Fix $B\ge2$ and $n\ge1$, and set
\begin{equation}\label{eq:abq}
 q=1-\frac1B,\qquad a=q^{-1/p},\qquad
 b=\frac{1-qa}{1-q}
   =\frac{1-q^{1-1/p}}{1-q}.
\end{equation}
Then $a>1$, and $0<b<1$. To see the latter inequalities, put $\alpha=1-1/p\in(0,1)$. Since $q<q^\alpha<1$,
\[
 0<1-q^\alpha<1-q.
\]
We also have the identities
\begin{equation}\label{eq:martingaleidentities}
 qa+(1-q)b=1,\qquad qa^p=1.
\end{equation}

Call a coordinate a success if it belongs to $\{0,\ldots,B-2\}$, and a failure if it equals $B-1$. For $u\in D_B^n$, let
\[
 \kappa(u)=\min\{j:u_j=B-1\},
\]
with the convention $\kappa(u)=n+1$ if there is no failure. Define
\begin{equation}\label{eq:explicitF}
 F(u)=
 \begin{cases}
  ba^{\kappa(u)-1},&\kappa(u)\le n,\\
  a^n,&\kappa(u)=n+1.
 \end{cases}
\end{equation}
This is a finite, nonnegative function on the uniform product space.

\begin{lemma}\label{lem:martingale}
For the function in \eqref{eq:explicitF}, the conditional expectations in \eqref{eq:conditional} satisfy
\begin{equation}\label{eq:explicitX}
 F_k(u_1,\ldots,u_k)=
 \begin{cases}
  ba^{\kappa(u)-1},&\kappa(u)\le k,\\
  a^k,&\kappa(u)>k.
 \end{cases}
\end{equation}
Moreover,
\begin{equation}\label{eq:moments}
 \E F^p=1+\frac nB b^p,\qquad
 \E F_*^p=1+\frac nB,\qquad
 H^p=q^{-n}.
\end{equation}
\end{lemma}

\begin{proof}
Start at $X_0=1$. After $k$ consecutive successes, the process is active with value $a^k$. At the next success, it becomes $a^{k+1}$. At a failure, it becomes $ba^k$ and remains at that value thereafter. Given an active state, the conditional mean of the next value is
\[
 q a^{k+1}+(1-q)ba^k=a^k
\]
by \eqref{eq:martingaleidentities}. A stopped state has unchanged conditional mean. Thus $(X_k)_{k=0}^n$ is a martingale for the coordinate filtration, and $X_n=F$. Iterating the conditional-mean relation backwards gives $X_k=\E(F\mid u_1,\ldots,u_k)=F_k$, proving \eqref{eq:explicitX}.

For $0\le k<n$, the event of a first failure after exactly $k$ successes has probability $q^k(1-q)$. On this event, $F=ba^k$. The running maximum is $F_*=a^k$, since the active values have increased through $1,a,\ldots,a^k$ and the process then drops to $ba^k<a^k$. On the event of $n$ successes, of probability $q^n$, both $F$ and $F_*$ equal $a^n$. Therefore
\begin{align*}
 \E F^p
 &=\sum_{k=0}^{n-1}q^k(1-q)b^pa^{kp}+q^na^{np}
 =n(1-q)b^p+1,\\
 \E F_*^p
 &=\sum_{k=0}^{n-1}q^k(1-q)a^{kp}+q^na^{np}
 =n(1-q)+1.
\end{align*}
The simplification uses $qa^p=1$ in every summand. The largest terminal value is $a^n$, so $H^p=a^{np}=q^{-n}$. Since $1-q=B^{-1}$, these are precisely \eqref{eq:moments}.
\end{proof}

\subsection{The limiting probability-space ratio}
For fixed $B$, the exact ratio of maximal to terminal moments is
\begin{equation}\label{eq:martingaleratio}
 \frac{\E F_*^p}{\E F^p}
 =\frac{1+n/B}{1+nb^p/B}.
\end{equation}
As $n\to\infty$, this tends to $b^{-p}$. As $B\to\infty$, $q=1-B^{-1}\to1$, and
\begin{equation}\label{eq:blimit}
 b=\frac{1-q^\alpha}{1-q}\longrightarrow\alpha=1-\frac1p.
\end{equation}
For example, \eqref{eq:blimit} follows from the mean value theorem applied to $t^\alpha$ on $[q,1]$. Hence the corresponding norm ratios approach $\alpha^{-1}=p'$. The next section selects all parameters explicitly within one family of real-line functions.

\section{Proofs of the main results}\label{sec:proof}

\subsection{An explicit family of compactly supported functions}
For each integer $B\ge2$, choose
\begin{equation}\label{eq:oneparameter}
 n=B^2,\qquad s=4,\qquad L=B^8,\qquad R=B^5.
\end{equation}
Let $F^{(B)}$ be the function in \eqref{eq:explicitF} with this base and depth. Write $b_B$ for the value of $b$ in \eqref{eq:abq}. Define, for $1\le j\le B^2$,
\begin{equation}\label{eq:explicitfamily}
 d_j^{(B)}(x)=\lfloor B\{B^{8(j-1)}x\}\rfloor,
 \qquad
 f_B(x)=F^{(B)}(d_1^{(B)}(x),\ldots,d_{B^2}^{(B)}(x))
         \ind_{[-B^5,B^5]}(x).
\end{equation}
Each $f_B$ is a nonzero, nonnegative, compactly supported step function.

Apply \eqref{eq:multiplicative} and \eqref{eq:moments}. Since $n/B=B$, $B^s/R=B^{-1}$, and $B^{1-s}=B^{-3}$, we obtain
\begin{equation}\label{eq:explicitlower}
 \frac{\norm{\HLcenter f_B}_p^p}{\norm{f_B}_p^p}
 \ge\left(1-\frac1B\right)
       \left(1-\frac2{B^3}\right)^{B^2}
       \frac{1+B}{1+Bb_B^p}.
\end{equation}
Bernoulli's inequality gives
\begin{equation}\label{eq:bernoulli}
 1-\frac2B
 \le\left(1-\frac2{B^3}\right)^{B^2}\le1.
\end{equation}
Thus both spatial factors in \eqref{eq:explicitlower} converge to one. By \eqref{eq:blimit},
\begin{equation}\label{eq:momentlimit}
 \frac{1+B}{1+Bb_B^p}
 =\frac{1+B^{-1}}{b_B^p+B^{-1}}
 \longrightarrow\left(1-\frac1p\right)^{-p}=(p')^p.
\end{equation}
The multiplicative estimate makes this simultaneous choice of parameters possible: the good-set factor tends to one even though the largest value of $F^{(B)}$ grows with $B$.

\subsection{Proof of Theorem \ref{thm:main}}
\begin{proof}
Proposition~\ref{prop:upper} gives $N_p\le p'$. Every $f_B$ from \eqref{eq:explicitfamily} is admissible in the definition of $N_p$, so \eqref{eq:explicitlower}--\eqref{eq:momentlimit} imply
\[
 (p')^p\le
 \liminf_{B\to\infty}\frac{\norm{\HLcenter f_B}_p^p}{\norm{f_B}_p^p}
 \le N_p^p\le(p')^p.
\]
All quantities are nonnegative, and taking $p$-th roots proves \eqref{eq:main}. In fact, the upper bound also shows that the norm ratios themselves converge:
\begin{equation}\label{eq:familylimit}
 \lim_{B\to\infty}\frac{\norm{\HLcenter f_B}_p}{\norm{f_B}_p}=p'.
\end{equation}

For an explicit normalization, \eqref{eq:exactdenominator} and \eqref{eq:moments} give
\begin{equation}\label{eq:normalization}
 \norm{f_B}_p^p=2B^5(1+Bb_B^p),\qquad
 g_B=[2B^5(1+Bb_B^p)]^{-1/p}f_B.
\end{equation}
Then $\norm{g_B}_p=1$ and $\norm{\HLcenter g_B}_p\to p'$.

It remains to obtain smooth nonnegative functions. Fix $\varepsilon>0$ and choose $B$ so that the ratio in \eqref{eq:familylimit} is greater than $p'-\varepsilon/2$. Let $\rho\in C_c^\infty(\R)$ be nonnegative with integral one, set $\rho_\delta(x)=\delta^{-1}\rho(x/\delta)$, and define $f_{B,\delta}=\rho_\delta*f_B$. Then $f_{B,\delta}$ is nonnegative and smooth with compact support, and $f_{B,\delta}\to f_B$ in $L^p$ as $\delta\to0$.

For locally integrable $u,v$, comparison of every averaging interval gives
\[
 |\HLcenter u-\HLcenter v|\le\HLcenter(u-v)
\]
wherever the quantities are finite. Applying Proposition~\ref{prop:upper},
\begin{equation}\label{eq:stability}
 \norm{\HLcenter f_{B,\delta}-\HLcenter f_B}_p
 \le p'\norm{f_{B,\delta}-f_B}_p\longrightarrow0.
\end{equation}
The denominator norms also converge to $\norm{f_B}_p>0$, so the norm ratios converge. A sufficiently small $\delta$ gives \eqref{eq:smoothsharp}.
\end{proof}

\subsection{Proof of Corollary \ref{cor:dyadic}}
\begin{proof}
The inequality $\HLdyad f\le\HLcenter f$ and Proposition~\ref{prop:upper} give the upper bound $p'$. For the lower bound, restrict the family \eqref{eq:oneparameter} to bases $B=2^m$, $m\ge1$. Its selected radii are
\begin{equation}\label{eq:dyadicradii}
 r_k=B^{4-8k}=2^{m(4-8k)},\qquad0\le k\le B^2,
\end{equation}
so every radius belongs to the family defining $\HLdyad$. The proof of \eqref{eq:explicitlower} used only these selected averages, and therefore it gives the same lower bound for $\norm{\HLdyad f_B}_p^p/\norm{f_B}_p^p$. Taking $B\to\infty$ along powers of two proves \eqref{eq:dyadicnorm}.
\end{proof}

\subsection{A proof using the additive estimate}
For completeness, the additive estimate gives a second route to the lower bound, with successive limits. Substituting \eqref{eq:moments} into \eqref{eq:additive} yields, for finite $B,n,s,R$,
\begin{equation}\label{eq:additiveconcrete}
 N_p^p\ge
 \left(1-\frac{B^s}{R}\right)
 \frac{1+n/B-2nB^{1-s}q^{-n}}{1+nb^p/B}.
\end{equation}
One first lets $R\to\infty$ with $B,n,s$ fixed, and then $s\to\infty$ with $B,n$ fixed. This gives $N_p^p\ge(1+n/B)/(1+nb^p/B)$. Sending $n\to\infty$ and then $B\to\infty$ gives $(p')^p$.

The order of these limits is dictated by the factor $q^{-n}$ in the additive error. The multiplicative estimate avoids this dependence and permits the simultaneous choices \eqref{eq:oneparameter}.

\section{Comparison with the critical-power constant}

We now compare the norm in Theorem~\ref{thm:main} with the critical-power constant \eqref{eq:peak}. Put $\alpha=1-1/p\in(0,1)$. Strict concavity of $t^\alpha$ gives, for $t>1$,
\[
 \frac{(t+1)^\alpha+(t-1)^\alpha}{2}<t^\alpha.
\]
Therefore
\[
 \frac{(t+1)^\alpha+(t-1)^\alpha}{2t}<t^{\alpha-1}<1.
\]
The expression on the left extends continuously to $t=1$, where it equals $2^{\alpha-1}<1$, and tends to zero at infinity. Its supremum is consequently strictly less than one. Hence
\begin{equation}\label{eq:strictpeak}
 c_p<p',\qquad1<p<\infty.
\end{equation}
Thus the unrestricted $L^p(\R)$ norm strictly exceeds the optimal constant on the convex class considered in \cite{GMM}. In particular, the critical-power value conjectured in \cite{DGS} does not give the norm on the full space.

The extremizing family \eqref{eq:explicitfamily} distributes its values across many separated scales. At the selected radii, the centered averages recover the conditional means of a finite martingale. The identity \eqref{eq:goodjointlaw} ensures that each configuration of the selected digits loses exactly the same proportion of centers. The lost proportion tends to zero under \eqref{eq:oneparameter}, while the martingale norm ratio tends to $p'$.

\section*{Acknowledgements} J.M. was partially supported by the Simons Foundation Grant \#453576. The author acknowledges the use of AI tools. 
 
\begingroup
\linespread{1}\selectfont

\endgroup
\end{document}